\documentclass[10 pt, leqno]{amsart}

\usepackage {amsfonts}
\usepackage{amsthm}
\usepackage{amscd}
\usepackage{amssymb}
\usepackage{latexsym}
\usepackage{amsmath}
\usepackage{mathrsfs}
\usepackage[small,nohug,heads=vee]{diagrams}
\usepackage{comment}

\theoremstyle{plain}
\newtheorem{tw}{Theorem}[section]

\newtheorem {lem} [tw]{Lemma}
\newtheorem {prop}[tw] {Proposition}

\newtheorem{cor}[tw]{Corollary}

\theoremstyle{definition}
\newtheorem {deft}[tw] {Definition}
\newtheorem {rem} [tw]{Remark}
\newtheorem {remark} [tw]{Remark}

\newcommand{\bc} {\Bbb C}
\newcommand{\bn}{\Bbb N}

\newcommand{\Com}{\Delta}
\newcommand{\Cou}{\epsilon}

\newcommand{\alg} {\mathsf{A}}

\newcommand {\id} {{\textrm{id}}}
\newcommand{\mlg} {\mathsf{M}}
\newcommand{\blg} {\mathsf{B}}

\newcommand{\nlg} {\mathsf{N}}
\newcommand{\wlg} {\mathsf{W}}
\newcommand{\jlg} {\mathsf{J}}
\newcommand{\QPn} {C(\mathbb{S}_n)}
\newcommand{\QPm} {C(\mathbb{S}_{n+1})}

\newcommand{\ilg} {\mathsf{I}}
\newcommand{\mclass}{\mathfrak{M}}
\newcommand{\Cclass}{\mathfrak{C}}
\newcommand{\Wclass}{\mathfrak{W}}

\newcommand{\QA}{\mathbb{S}}
\newcommand{\tu}{\textup}

\newcommand{\Hil}{\mathsf{H}}
\newcommand{\hil}{\mathsf{h}}

\newcommand{\Ker}{\tu{Ker}}

\newcommand{\Proj}{\mathcal{P}}

\newcommand{\ot}{\otimes}

\newcommand{\wot}{\overline{\otimes}}

\newcommand{\wt}{\widetilde}

\numberwithin{equation}{section}



\dedicatory{Dedicated to Stanis\l aw Lech Woronowicz on the occasion of his 70th birthday}

\begin{document}

\author{Debashish Goswami}
\address{Stat-Math Unit, Indian Statistical Institute, 203, B.\ T.\ Road, Kolkata 700 208, India} \email{goswamid@isical.ac.in}
\author{Adam Skalski}
\address{  Institute of Mathematics of the Polish Academy of Sciences,
ul.\'Sniadeckich 8, 00-956 Warszawa, Poland} \email{a.skalski@impan.pl}

\begin{abstract}
Two different models for a Hopf--von Neumann algebra of bounded functions on the quantum  semigroup of all (quantum) permutations of infinitely many elements are
proposed, one based on projective limits of enveloping von Neumann algebras related to finite quantum permutation groups, and the
second on a universal property with respect to infinite magic unitaries.
\end{abstract}

\title{\bf  On two possible constructions of the quantum
semigroup of all quantum permutations of an infinite countable set}

\maketitle


Classical groups first entered mathematics as collections of all symmetries of a given object, be it a finite set, a polygon, a
metric space or a manifold. Original definitions of quantum groups (also in the topological context, see \cite{woronowicz98} and
\cite{kuv}) had rather algebraic character. Recent years however have brought many developments in the theory of quantum symmetry
groups, i.e.\ quantum groups defined as universal objects acting (in the sense of quantum group actions) on a given structure.
The first examples of that type were introduced in \cite{Wang}, where S.\,Wang defined the quantum group of permutations of a
finite set, $\QA_n$. It turns out that the $C^*$-algebra of `continuous functions on a quantum permutation group of $n$ elements',
$C(\QA_n)$, is generated by entries of a universal $n$ by $n$ magic unitary, i.e.\ a unitary matrix whose entries are orthogonal
projections. Later the theory was extended to quantum symmetry groups of finite graphs (\cite{graph}), finite metric spaces
(\cite{metric}), $C^*$-algebras equipped with orthogonal filtrations (\cite{orthfiltr}), and quantum isometry groups of compact noncommutative manifolds (\cite{Deb}). In all these cases the
structure whose (quantum) symmetries are studied has finite or compact flavour, so that the resulting quantum symmetry group is
compact.

In this paper we study possible definitions of the quantum permutation (semi)group of an infinite countable set. Even in the classical
context there is a natural choice here  -- we can either consider the group of all permutations of $\bn$, $\tu{Perm}(\bn)$, or the
group of all `finite range' permutations of $\bn$, usually denoted by $S_{\infty}$. From the analytic point of view the second group
arises more naturally, as it is a direct limit of finite permutation groups $S_n$. Hence this will be the group whose quantum
version we want to discuss here.   As on the level of groups we have embeddings $S_n \hookrightarrow S_{n+1}$, on the level of
algebras we obtain surjective morphisms $C(S_{n+1}) \twoheadrightarrow C(S_{n})$. Therefore it is natural to expect that the algebra
of continuous functions on the quantum version of $S_{\infty}$ will arise as the inverse (projective) limit of algebras
$C(\QA_n)$  -- note however that the situation here is more complicated than in the classical framework, as quantum groups $\QA_n$ are neither finite nor discrete for $n\geq 4$.
Moreover  projective limits of $C^*$-algebras do not behave well, which is easy to understand even in the commutative
setting: a direct limit of locally compact spaces need not be locally compact. Hence one either needs to consider pro-$C^*$-algebras, as suggested in a slightly different context in a recent paper
(\cite{Varghese}), or, as we do here, work with von Neumann algebras. Precisely speaking,
we construct in this note the algebra $\wlg_{\infty}$, a candidate for $L^{\infty}(\QA_{\infty})$, as the limit of
enveloping von Neumann algebras of $C(\QA_n)$ and  study its universal properties. Another possible approach to infinite
quantum permutation groups exploits the fact that the algebras $C(\QA_n)$ are defined in terms of universal magic unitaries, so
by analogy one can investigate a universal von Neumann algebra generated by entries of an infinite magic unitary. We show that
such an algebra exists and is a proper subalgebra of $\wlg_{\infty}$. In both cases the algebras in question come equipped with a natural comultiplication.
We do not know if either of the resulting Hopf--von Neumann algebras fits into the theory of locally compact quantum groups developed in \cite{kuv}; they admit (bounded) antipodes, but the existence of invariant weights is not known.

The detailed plan of the paper is as follows: in Section 1 we discuss projective limits of von Neumann algebras; although these
results are not difficult and can be deduced from the corresponding statements for Banach spaces (\cite{Semadeni}), we could not locate a specific reference to the von Neumann algebra setting, where the explicit structure of the projective limit is easier to see (and will be used in Section 3 of the paper). We also include several
lemmas on extending maps to the projective limits. A short Section 2 contains applications of these results to projective limits
of Hopf-von Neumann algebras. In Section 3 we recall basic facts on Wang's quantum permutation groups and describe the first of
two possible candidates for the algebra $L^{\infty}(\QA_{\infty})$, constructed as the projective limit of the enveloping von
Neumann algebras of $C(\QA_n)$. In Section 4 we propose an alternative approach in terms of a universal `infinite magic unitary'
and explain why this leads to a different Hopf--von Neumann algebra.

The spatial tensor of $C^*$-algebras will be denoted $\ot$, and the ultraweak tensor product of von Neumann algebras  $\wot$. For
 a von Neumann algebra $\mlg$ its lattice of projections will be denoted $\Proj(\mlg)$ and the central carrier of
$p\in \Proj(\mlg)$ (i.e.\ the smallest projection in $Z(\mlg)$ dominating $p$) will be denoted $z(p)$.

\section{Projective limits of von Neumann algebras}

In this section we define, establish existence and prove basic properties of projective  limits of von Neumann algebras. The statements and the ideas of proofs follow the pattern established for example in \cite{Semadeni}, but the nature of the weak$^*$-closed ideals in von Neumann algebras make it possible to describe the resulting structures explicitly. Although
the theorems remain valid for general directed index sets, we consider only projective systems indexed by $\bn$. Note that several categorical theorems related to von Neumann algebras (with main focus on the abstract properties of the tensoring procedure, but also describing for example inductive limit constructions) can be found in \cite{Guichardet}.

\newpage

\begin{deft} \label{classes}
A sequence  $(\mlg_n)_{n \in \bn}$ is a \emph{projective system of  von Neumann algebras} if it is a sequence of von Neumann
algebras equipped with surjective normal $^*$-homomorphisms $\phi_n:\mlg_{n+1} \to \mlg_n$ (the maps $\phi_n$ form a part of the
definition, but we omit them from the notation). Define the following class of von Neumann algebras: $\mclass=\{\mlg: \forall_{n
\in \bn} \;\exists\; \psi_n:\mlg \to \mlg_n$, a surjective normal  $^*$-homomorphism such that $\psi_n = \phi_n \circ
\psi_{n+1}\}$. We say that $\mlg \in \mclass$ is a final object (in other words a colimit) for $\mclass$ if for each $\nlg \in \mclass$ there exists a
surjective normal morphism $\psi: \nlg \to \mlg$ such that $\psi^{(\mlg)}_{n} \circ \psi= \psi^{(\nlg)}_n$ for all $n \in \bn$.
\end{deft}

Note that it is not clear at the moment whether even if a final object for $\mclass$ exists, it is unique.

\begin{tw} \label{projective}
Let $(\mlg_n)_{n \in \bn}$ be a projective system of von Neumann algebras. Then the class $\mclass$ admits a (unique) final
object.
\end{tw}
\begin{proof}
The construction is based on the properties of weak$^*$-closed two-sided ideals in von Neumann algebras.  Let $n \geq 2$ and
$\ilg_{n}= \Ker (\phi_{n-1})$. Let $r_n \in \Proj(Z(\mlg_{n}))$ be the projection such that $\ilg_{n} = r_{n} \mlg_{n}$ (recall
that $r_n:= \sup \{p \in \Proj(\mlg_{n-1}): \phi_n(p) =0\}$). A well-known (and easy to check) fact states that the map
$\phi_{n-1}|_{r_n^{\perp} \mlg_n}:r_n^{\perp} \mlg_n \to \mlg_{n-1}$ is an isomorphism. Let $\blg_n = r_n \mlg_{n}$ and define
additionally $\blg_1 = \mlg_1$. Then each $\mlg_{n}$ has a natural decomposition of the form $\mlg_{n} = \oplus_{k=1}^n \blg_k$,
and additionally this decomposition is `well behaved' with respect to the maps $\phi_n$. Not surprisingly, the final object in
$\mclass$ will be isomorphic to $\prod_{n=1}^{\infty} \blg_n$. Below we give a detailed proof of this fact.

Observe first that the class $\mclass$ is non-empty. Indeed, define $\mlg_{\infty}=\{(m_n)_{n=1}^{\infty} \in
\prod_{n=1}^{\infty} \mlg_n: \phi_n(m_{n+1})=m_n\}$. Then $\mlg_{\infty}$ is a weak$^*$-closed 
subalgebra of
$\prod_{n=1}^{\infty} \mlg_n$, hence a von Neumann algebra. It is clear that the projections on the individual coordinates are normal
$^*$-homomorphisms; they satisfy the intertwining relation with $\phi_n$ by construction. Surjectivity follows from the existence
of isometric lifts for selfadjoint elements in $C^*$-algebras (hence bounded lifts for arbitrary elements of $\mlg_n$ to elements
in $\mlg_{\infty})$. In fact $\mlg_{\infty}$ will be (isomorphic to) the final object for $\mclass$.

Let $\nlg \in \mclass$ and denote by $\jlg_n$ the kernel of the corresponding map $\psi_n:\nlg \to \mlg_n$. Let $w_n \in
\Proj(Z(\nlg)))$ be the projection such that $\jlg_n = w_n \nlg$. As in the first part of the proof, $\psi_n|_{w_n^{\perp}
\nlg}:w_n^{\perp} \nlg \to \mlg_n$ is an isomorphism. Write $z_n:=w_n^{\perp}$. As $\jlg_{n+1} \subset \jlg_n$, the sequence
$(z_n)_{n=1}^{\infty}$ is increasing. Define additionally $z_{\infty}=\lim_{n \to \infty} z_n$, $p_1=z_1$ and  $p_n=z_n -
z_{n-1}$ for $n \geq 2$, so that $z_{\infty} = \sum_{n=1}^{\infty}p_n$. As all projections $p_n$ are central, we obtain a natural
increasing sequence of von Neumann algebras $\oplus_{k=1}^n p_k \nlg$ whose union is weak$^*$-dense in $z_{\infty} \nlg$. It is
easy to see that this yields a natural isomorphism $z_{\infty} \nlg \approx \prod_{n=1}^{\infty} p_n \nlg$.

Note that $z_{\infty} \nlg \in \mclass$ - indeed, the only thing to check is that the maps $\psi_n|_{z_{\infty} \nlg}: \mlg_n$
are surjections, and this follows from the stated above surjectivity of $\psi_n|_{z_n \nlg}$. Our claim is that $z_{\infty} \nlg$
is the final object of $\mclass$. Indeed, it suffices to show that if $\wlg$ is another von Neumann algebra in $\mclass$, then
$z_{\infty}^{(\wlg)} \wlg$ is isomorphic to $z_{\infty} \nlg$ and the isomorphism intertwines the corresponding maps into
$\mlg_n$. For the first statement it suffices to describe the algebras $p_n \nlg$ in terms of the projective sequence with which
we started. Let $n \geq 2$. Consider the diagram
\begin{diagram}
&&    &  \;z_{n-1} \nlg \oplus p_n \nlg &\;\; =& z_n \nlg \\
&&\ldTo_{\psi_{n-1}|_{z_{n-1} \nlg}} & & & \\
& \mlg_{n-1} &&                 &  &      \dTo_{\psi_n|_{z_n \nlg}}    \\
&&\luTo^{\phi_{n-1}|_{r_n^{\perp} \mlg_n}} & & & \\
 &     & & \;\;\;r_n^{\perp} \mlg_n \oplus \blg_n &\;\; = & \mlg_n
\end{diagram}
in which all arrows are isomorphisms. It immediately implies that $p_n \nlg$ is isomorphic to $\blg_n$ (note that for $n=1$ this
also holds). Moreover looking at the diagram above we see that if we denote the corresponding isomorphism between $p_n \nlg$ and
$\blg_n$ by $\gamma_n$, we can check inductively that $\gamma_1 \oplus \cdots \oplus \gamma_n: z_n \nlg \to \mlg_n$ coincides
with $\psi_n$, which assures that the natural isomorphism between $z_{\infty}^{(\wlg)}  \wlg$ and $z_{\infty} \nlg$ intertwines
the respective $\psi_n$ and $\psi^{(\wlg)}_n$ maps.

We can check that for $\nlg:=\mlg_{\infty}$ we have $z_{\infty} = 1_{\mlg_{\infty}}$. Indeed, if  $(m_n)_{n=1}^{\infty} \in
w_{\infty}M_{\infty}$ then $(m_n)_{n=1}^{\infty} \in \Ker (\psi_n)$ for each $n \in \bn$, so $(m_n)_{n=1}^{\infty}=0$.

It remains to prove uniqueness. Suppose then that $\nlg$ is a final object in $\mclass$ and let $\wlg$ be a final object in
$\mclass$ constructed above. Note that if $\psi^{(\wlg)}_n:\wlg \to \mlg_n$ denote the usual surjections, the construction above
implies that $\bigcap_{n=1}^{\infty}\Ker (\psi^{(\wlg)} _n)= \{0\}$. There is a surjective map $\psi: \wlg \to \nlg$ such that
$\psi_n = \psi^{(\wlg)}_n \circ \psi$ for all $n \in \bn$. Thus we must have $\bigcap_{n=1}^{\infty} \Ker (\psi_n) = \{0\}$, or
equivalently $z_{\infty}= 1_{\nlg}$, where $z_{\infty}$ is constructed for $\nlg$ as above. Then $\nlg = \nlg z_{\infty}$ and the
arguments above show that $\nlg \approx \wlg$.
\end{proof}

\begin{deft}
Let $(\mlg_n)_{n \in \bn}$ be a projective system of von Neumann algebras. The final object in the class $\mclass$ will be called
the projective limit of $(\mlg_n)_{n \in \bn}$ and denoted $\mlg_{\infty}$.
\end{deft}

In the next section we will show that if $(\mlg_n)_{n \in \bn}$ is a projective system of Hopf--von Neumann algebras, with the
normal surjections $\phi_n$ intertwining the respective coproducts, then $\mlg_{\infty}$ has a natural Hopf--von Neumann algebra
structure. To this end  we present here several lemmas related to constructing maps acting on/to/between projective limits.

\begin{lem} \label{iotas}
Let $(\mlg_n)_{n \in \bn}$ be as in Theorem \ref{projective} and let us adopt the notations in the proof that theorem. Define
additionally for each $n \in \bn$ the map $\iota_n:\mlg_n \to \mlg_{\infty}$ to be the inverse of $\psi_n|_{z_n \mlg_{\infty}}$
(or more precisely the composition of that inverse with the embedding of $z_n \mlg_{\infty}$ into $\mlg_{\infty}$)). Then we have
the following: for each $n \in \bn$, $x \in \mlg_{n+1}$
\[ \iota_{n} (\phi_n (x)) = z_n \iota_{n+1} (x) \]
\end{lem}
\begin{proof}
It is a direct consequence of the diagram above, this time interpreted as follows:
 \begin{diagram}
&&    &  \;z_{n} \mlg_{\infty} &\lTo^{z_n \cdot}&  \mlg_{\infty} \\
&&\ruTo_{\iota_{n}} & & & \\
& \mlg_{n} &&                 &  &      \uTo_{\iota_{n+1}}    \\
&& \luTo(3,2)^{\phi_n} && & \\
 &     & & & & \mlg_{n+1}
\end{diagram}
- note that now the maps are not necessarily isomorphisms.\end{proof}

\begin{lem} \label{homext}
Suppose that $(\nlg_n)_{n=1}^{\infty}$, $\wlg$ are von Neumann algebras and that $\nlg = \prod_{n\in\bn} \nlg_n$. For each $n \in
\bn$ denote the central projection in $\nlg$ corresponding to $\nlg_n$ by $p_n$. Let (for each $n\in \bn$)  $\kappa_n:\wlg \to
\prod_{k=1}^n \nlg_k$ be a normal contractive map and suppose that (for each $ w \in \wlg, n \in \bn$)
\begin{equation} \label{kapcom}\kappa_{n} (w) = \sum_{k=1}^n p_k \kappa_{n+1}(w).\end{equation}
Then there exists a unique normal
contraction $\kappa:\wlg \to \nlg$ such that
\[ \kappa_{n} (w) = \sum_{k=1}^n p_k \kappa(w).\]
If each $\kappa_n$ is a $^*$-homomorphism (respectively, a $^*$-antihomomorphism), $\kappa$ is also $^*$-homomorphic
(respectively, $^*$-antihomomorphic).
\end{lem}

\begin{proof}
Let $w\in \wlg$. Define
\[\kappa(w) = \sum_{n=1}^{\infty} p_n \kappa_n(w) = \lim_{n \to \infty} \kappa_n(w).\]
The equality of both expressions follows from the formula \eqref{kapcom} and the properties of weak$^*$ topology in $\nlg$
(recall that we have a natural Banach space isomorphism $\nlg_* \approx \bigoplus_{n=1}^{\infty} (\nlg_n)_*$, where the last sum
is of the $l^1$-type). Similarly, normality of $\kappa$  follows from the explicit description of the predual of $\nlg$ and
normality of each $\kappa_n$. The statement on algebraic properties of $\kappa$ is easy to check, and the uniqueness is clear.
\end{proof}

The last two results have a following consequence.

\begin{prop} \label{antip}
Suppose that $(\mlg_n)_{n\in \bn}$ and $(\nlg_n)_{n\in \bn}$ are projective systems of von Neumann algebras, with connecting maps
respectively denoted by $(\phi_n^{(\mlg)})_{n\in \bn}$ and $(\phi_n^{(\nlg)})_{n\in \bn}$ and the maps from the final objects $\mlg_{\infty}$ and $\nlg_{\infty}$ respectively denoted by $(\psi_n^{(\mlg)})_{n\in \bn}$ and $(\psi_n^{(\nlg)})_{n\in \bn}$. Let $\lambda_n:\mlg_n \to \nlg_n$
($n \in \bn$) be normal contractive maps such that
\[ \lambda_{n} \circ \phi_n^{(\mlg)} = \phi_n^{(\nlg)} \circ \lambda_{n+1}, \;\;\;n \in \bn.\]
Then there exists a unique map $\lambda_{\infty}:\mlg_{\infty} \to \nlg_{\infty}$ such that
\[ \lambda_{n} \circ \psi_n^{(\mlg)} = \psi_n^{(\nlg)} \circ \lambda_{\infty}, \;\;\; n \in \bn.\]
If each $\lambda_n$ is a $^*$-homomorphism (respectively, a $^*$-antihomomorphism, a unital map), $\lambda$ is also
$^*$-homomorphic (respectively, $^*$-antihomomorphic, unital).
\end{prop}
\begin{proof}
Use the notation of Theorem \ref{projective} and Lemma \ref{iotas}, adorning respective maps with $^{(\mlg)}$ and $^{(\nlg)}$.
Define $\wt{\lambda}_n:\mlg_{\infty} \to z_n^{(\nlg)} \nlg_{\infty}$ ($n \in \bn$) as $\wt{\lambda}_n = \iota_n^{(\mlg)} \circ
\lambda_n \circ \psi_n^{(\nlg)}$. Then
\begin{align*} z_n^{(\nlg)}
\wt{\lambda}_{n+1} (\cdot) &= z_n^{(\nlg)} (\iota_{n+1}^{(\nlg)} \circ \lambda_{n+1} \circ \psi_{n+1}^{(\mlg)})(\cdot) =
\iota_{n}^{(\nlg)} \circ \phi_n^{(\nlg)} \circ \lambda_{n+1} \circ \psi_{n+1}^{(\mlg)} \\& = \iota_{n}^{(\nlg)} \circ \lambda_{n}
\circ \phi_n^{(\mlg)} \circ \psi_{n+1}^{(\mlg)} = \iota_{n}^{(\nlg)} \circ \lambda_{n} \circ \psi_{n}^{(\mlg)} =
\wt{\lambda}_n,\end{align*} where in the second equality we used Lemma \ref{iotas}. Apply now Lemma \ref{homext} for
$\kappa_n:=\lambda_n$, $\wlg:=\mlg_{\infty}$ and $\nlg_n:=p_n^{(\nlg)} \nlg_{\infty}$. This yields a map
$\lambda_{\infty}:\mlg_{\infty} \to \nlg_{\infty}$ such that
\[\wt{\lambda}_n = z_n^{(\nlg)} \lambda_{\infty}(\cdot).\]
Straightforward identifications using the commuting diagrams presented earlier end the proof of the main statement. As before,
uniqueness and algebraic properties of $\lambda_{\infty}$ follow easily.
\end{proof}

The above lemma provides a simple corollary describing a construction of maps acting from $\mlg_{\infty}$ into some other von
Neumann algebra.

\begin{cor} \label{homext2}
Let $(\mlg_n)_{n \in \bn}$ be a projective system of von Neumann algebras; adopt the notations of Theorem \ref{projective}.  
 Let (for each $n\in \bn$)  $\mu_n:\mlg_n \to \wlg$ be a
normal $^*$-homomorphism and suppose that (for each $n \in \bn$)
\[\mu_{n} \circ \phi_n = \mu_{n+1}.\]   Then there exists a
unique normal $^*$-homomorphism $\mu: \mlg_{\infty}\to \wlg$ such that
\[ \mu = \mu_{n} \circ \psi_n.\]
\end{cor}
\begin{proof}
It suffices to apply Proposition \ref{antip} to the projective systems $(\mlg_n)_{n \in \bn}$ and $(\nlg_n)_{n \in \bn}$, where
$\nlg_n:=\wlg$ and $\phi_n=\id_{\wlg}$ for all $n \in \bn$ .
\end{proof}

\section{Projective limits of Hopf--von Neumann algebras}

Here we apply the results of Section 1 to construct the projective limit of a projective sequence of Hopf--von Neumann algebras.

\begin{deft}
A Hopf--von Neumann algebra is a von Neumann algebra equipped with a coproduct, i.e.\ a unital normal $^*$-homomorphism $\Com:
\mlg \to \mlg \wot \mlg$ which is coassociative:
\[ (\id_{\mlg} \ot \Com) \Com = (\Com \ot \id_{\mlg} ) \Com.\]
\end{deft}

\begin{deft}
A sequence $(\mlg_n)_{n \in \bn}$ is called a projective system of Hopf--von Neumann algebras if it is a projective system of von
Neumann algebras, each $\mlg_n$ is a Hopf--von Neumann algebra (with the coproduct $\Com_n: \mlg_n \to \mlg_n \ot \mlg_n$) and the
surjective normal homomorphisms $\phi_n:\mlg_{n+1} \to \mlg_n$ satisfy the conditions
\[ (\phi_n \ot \phi_n) \Com_{n+1} = \Com_n \phi_n.\]
\end{deft}

\begin{tw} \label{projcop}
Let $(\mlg_n)_{n \in \bn}$ be a projective system of Hopf--von Neumann algebras.  Then $\mlg_{\infty}$ is also a Hopf--von Neumann
algebra: there exists a unique coproduct $\Com:\mlg_{\infty} \to \mlg_{\infty}\wot \mlg_{\infty}$ such that
\begin{equation}
 \label{cop}
 \Com_n  \psi_n = (\psi_n \ot \psi_n) \Com, \;\; n \in \bn.\end{equation}
 In addition if each $\Com_n$ is injective, so is $\Com$.
\end{tw}

\begin{proof}
Observe that the sequence $(\mlg_n \wot \mlg_n)_{n \in \bn}$, together with surjective connecting maps $\phi_n \ot \phi_n:
\mlg_{n+1} \wot \mlg_{n+1} \to \mlg_n \wot \mlg_n$ forms a projective limit of von Neumann algebras; moreover a projective limit
of this sequence can be easily identified with $\mlg_{\infty} \wot \mlg_{\infty}$. Hence an application of Proposition
\ref{antip} yields the existence and uniqueness of a unital normal $^*$-homomorphism $\Com: \mlg_{\infty} \to \mlg_{\infty} \wot
\mlg_{\infty}$ satisfying \eqref{cop}.

Coassociativity of $\Com$ can be proved in an analogous way, exploiting the uniqueness part of Proposition \ref{antip}.

If each $\Com_n$ is injective, $x \in \mlg_{\infty}$ and $\Com (x)=0$, then by \eqref{cop} we have (for each $n \in \bn$)
$\psi_n(x)=0$. Via identifications in Theorem \ref{projective} we see that $z_n x =0$ for all $n \in \bn$, which implies that
$x=0$.
\end{proof}

We could also consider Hopf--von Neumann algebras with a \emph{counit}, i.e.\ a normal character $\Cou:\mlg\to \bc$ such that
\[ (\Cou \ot \id_{\mlg}) \Com =  (\id_{\mlg} \ot \Cou )\Com = \id_{\mlg}\]
Then for $(\mlg_n)_{n \in \bn}$ to be  a projective system of Hopf--von Neumann algebras  we additionally require that 
\[\Cou_{n}  \circ \phi_n= \Cou_{n+1} , \;\; n\in \bn.\]
Lemma \ref{homext2} and a simple calculation imply that if the above conditions are satisfied, then $\mlg_{\infty}$ admits a
natural counit. 

We finish this section with a short discussion of projective limits of actions of Hopf--von Neumann algebras.

\begin{deft}
Let $\wlg$ be a von Neumann algebra and $(\mlg, \Com)$ be a Hopf--von Neumann algebra. We say that $\alpha: \wlg \to \wlg \wot
\mlg$ is a (Hopf--von Neumann algebraic) action of $\mlg$ on $\wlg$ if it is a normal unital injective $^*$-homomorphism such that
\[ (\id_{\wlg} \ot \Com) \alpha = (\alpha \ot \id_{\mlg} ) \alpha.\]
\end{deft}

A combination of Theorem \ref{projcop} and Lemma \ref{homext} yields the following result, which says that the Hopf--von Neumann
algebraic actions behave well under passing to projective limits.

\begin{tw} \label{projact}
Let $\wlg$ be a von Neumann algebra and let $(\mlg_n)_{n \in \bn}$ be a projective system of Hopf--von Neumann algebras. Denote by $\mlg_{\infty}$ the Hopf--von Neumann algebra arising as the projective limit in the sense of Theorem \ref{projcop}. Let $(\alpha_n)_{n \in \bn}$
be a sequence of actions of $\mlg_n$ on $\wlg$ such that for each $n \in \bn$
\[ (\id_{\wlg} \ot \phi_n) \alpha_{n+1} =  \alpha_n,\]
where $\phi_n$ are connecting maps defining the system $(\mlg_n)_{n \in \bn}$. Then there exists a unique action $\alpha$ of
$\mlg_{\infty}$ on $\wlg$ such that for each $n \in \bn$
\[ (\id_{\wlg} \ot \psi_n) \alpha =  \alpha_n.\]
\end{tw}
\begin{proof}
Similar to that of Theorem \ref{projcop}, using the fact that the von Neumann algebra $\wlg \wot \mlg_{\infty}$ is the projective
limit of the system $(\wlg \wot\mlg_n)_{n \in \bn}$, with the connecting maps $\id_{\wlg} \ot \phi_n$, and then applying
Proposition \ref{antip}.
\end{proof}

\section{The Hopf--von Neumann algebra of `all finite quantum permutations of an infinite set' as a projective limit}
\label{winfty}

Let $C(\mathbb{S}_n)$ denote the algebra of continuous functions on the quantum permutation group of the $n$-point set. Recall
(\cite{Wang}) that it is the universal $C^*$-algebra generated by the collection of orthogonal projections
$\{p_{ij}:i,j=1,\ldots,n\}$ such that for each $i=1, \ldots,n$ there is $\sum_{j=1}^n p_{ij} = \sum_{j=1}^n p_{ji} =1$. The
coproduct, counit and (bounded, $^*$-antihomomorphic) antipode are defined on $C(\mathbb{S}_n)$ by the formulas  ($i,j =1,
\ldots,n$)
\[ \Com_n(p_{ij}) = \sum_{k=1}^n p_{ik} \ot p_{kj},\]
\[ \Cou_n(p_{ij}) = \delta_{ij},  \;\;\; \kappa_n(p_{ij}) = p_{ji}.\]
For more properties of $C(\mathbb{S}_n)$ and its connections to combinatorics, free probability, Hadamard matrices and other problematics we refer to the surveys
\cite{Teosurvey} and \cite{surveyTeo}.
Denote the enveloping von Neumann algebra of $C(\mathbb{S}_n)$ by $\wlg_n$. Standard arguments show that maps $\Com_n$, $\Cou_n$
and $\kappa_n$ have unique normal extensions to $\wlg_n$, which will be denoted by the same symbols -- so that for example
$\Com_n:\wlg_n \to \wlg_n \wot \wlg_n$.

For each $n \in \bn$ we denote by $\omega_n$ the  natural surjection (and a compact quantum group morphism) from
$C(\mathbb{S}_{n+1})$ to $C(\mathbb{S}_n)$, which corresponds to mapping $\begin{bmatrix} P & 0 \\ 0& 1\end{bmatrix} \to P$ and
whose existence follows from the universal properties. This induces in a standard way the surjection on the level of universal
enveloping von Neumann algebras (it is enough to define $\phi_n=\omega_n^{**}:\QPm^{**} \to \QPn^{**}$ - the fact that $\phi_n$ is
multiplicative is the standard consequence of the definition of the Arens multiplication, surjectivity follows from the fact that
images of normal representations of von Neumann algebras are ultraweakly closed). Hence the sequence of algebras
$(\wlg_n)_{n=1}^{\infty}$ forms a projective system of von Neumann algebras.  As $\omega_n$ intertwined the respective coproducts
on the level of $C^*$-algebras, so does $\phi_n$ on the level of von Neumann algebras; similarly $\Cou_{n+1} \circ \phi_n =
\Cou_n$ for all $n \in \bn$. Hence Theorem \ref{projcop} implies that the projective limit of $(\wlg_n)_{n\in \bn}$ is a Hopf--von
Neumann algebra, denoted further by $\wlg_{\infty}$. We formulate it as a theorem:

\begin{tw}
The sequence  $(\wlg_n:=C(\mathbb{S}_n)^{**})_{n=1}^{\infty}$ is a projective system of Hopf--von Neumann algebras with counits.
Hence its projective limit denoted by $\wlg_{\infty}$ is also a Hopf--von Neumann algebra with a counit.
\end{tw}

\begin{proof}
A direct consequence of Theorem \ref{projcop} and the discussion before the theorem.
\end{proof}

In general we cannot expect  Hopf--von Neumann algebras to possess antipodes. Here
we have however the following fact.

\begin{tw}
The Hopf--von Neumann algebra $\wlg_{\infty}$ admits a unique $^*$-antihomo-morphic involutive map $\kappa:\wlg_{\infty} \to
\wlg_{\infty}$ such that
\[  \kappa_n \circ \psi_n =  \psi_n \circ\kappa, \;\;\; n \in \bn,\]
where $\psi_n:\wlg_{\infty} \to \wlg_n$ are the canonical surjections.
\end{tw}

\begin{proof}
As $\omega_n \circ \kappa_{n+1}|_{\QPm} = \kappa_n \circ \omega_n|_{\QPm}$, we also have a similar relation on the level of maps
between the enveloping von Neumann algebras, with $\omega_n$ replaced by $\phi_n$. Hence Proposition \ref{antip} implies the
existence of the $^*$-antihomomorphic map $\kappa$ as above; the fact it is involutive is a consequence of the analogous property of all $\kappa_n$.
\end{proof}

It would be of course more natural to use for the projective limit construction instead of $C(\mathbb{S}_n)^{**}$ the algebras
$L^{\infty}(\mathbb{S}_n)$, the von Neumann completions of $C(\mathbb{S}_n)$ in the GNS representation with respect to the
respective Haar states. The problem lies in the fact that the maps $\omega_n$ cannot extend to `reduced' versions of the algebras
of $C(\mathbb{S}_n)$, so also not to normal continuous maps $L^{\infty}(\mathbb{S}_{n+1}) \to L^{\infty}(\mathbb{S}_n) $. The first
statement is a consequence of the fact that $C(\mathbb{S}_n)$ is not \emph{coamenable} for $n \geq 5$, as follows from the
quantum version of the Kesten criterion for amenability (\cite{Kesten}).

The fact that we can only construct the projective limit using the universal completions is related to the problem described in
the next remark.

\begin{remark}
Recently C.\,K\"ostler and R.\,Speicher introduced
 a notion of \emph{quantum exchangeability} or \emph{invariance under quantum permutations} for a family of
quantum random variables (see Definition 2.4 in \cite{ClausRoland}). This notion was later studied by S.\,Curran in \cite{Curran}
and extended to finite sequences; the basic idea is that a sequence of random variables is quantum exchangeable if its
distribution (understood as a state on a von Neumann algebra generated by the variables in question) is invariant under natural
actions of all Wang's quantum permutation groups $\QA_n$. Classically exchangeability can be defined as the invariance of the
distribution under the action of the infinite permutation group; it would be natural to expect a similar result in the quantum
context. It is not clear whether our definition would allow such a formulation; although Theorem \ref{projact} offers a way of
constructing actions of the projective limit, the natural actions of quantum permutation groups considered in \cite{ClausRoland}
are defined only on the Hopf algebraic level. As shown in Theorem 3.3 of \cite{Curran} (see also Section 5.6 of that paper), in
the presence of quantum exchangeability the actions can be extended to the reduced von Neumann algebraic completions
$L^{\infty}(\mathbb{S}_n)$, but to apply Theorem \ref{projact} to obtain the action of $\wlg_{\infty}$ on the von Neumann algebra
in question we would need to be able to extend the original actions to $C(\QA_n)^{**}$.

\end{remark}

\section{Universal von Neumann algebra generated by an infinite magic unitary}

In this section we shall define a quantum analogue of the algebra of functions on the permutation group of a countably infinite
set as the universal von Neumann algebra generated by the entries of an `infinite magic unitary'.

We begin with a $C^*$-algebraic construction.

\begin{deft}
Let $\Cclass$ denote the category with objects $({\mathsf C}, \{ q_{ij}, i,j=1,...,\infty\}),$ where ${\mathsf C}$ is a (possibly
nonunital) $C^*$-algebra generated by a family of orthogonal projections  $\{q_{ij}: i,j\in \bn\}$ and such that there exists a
faithful (and nondegenerate) representation $(\pi, \Hil)$ of $\mathsf{C}$ such that for each $i\in \bn$
\begin{equation}
\sum_{j=1}^{\infty} \pi( q_{ij}) = \sum_{j=1}^{\infty} \pi( q_{ji}) =1_{B(\Hil)}, \label{infperm}\end{equation} with the
convergence understood in the strong operator topology. A morphism from $({\mathsf C},\{ q_{ij}\})$ to $({\mathsf C}^\prime, \{
q^\prime_{ij} \})$ is given by a (necessarily nondegenerate) $C^*$-homomorphism from ${\mathsf C}$ to ${\mathsf C}^\prime$ which
maps $q_{ij}$ to $q^\prime_{ij}$ for all $i,j\in \bn$.
\end{deft}

\begin{tw}
\label{construction}
 The  category $\Cclass$ has a universal (initial) object.
\end{tw}
\begin{proof}
Consider  the (formal) $\ast$-algebra ${\mathcal B}$ generated by symbols
 $\{b_{ij}:i,j \in \bn\}$ which are selfadjoint idempotents 
 \begin{equation} \label{bproj}  b_{ij}=b_{ij}^*=b_{ij}^2, \end{equation}
 and satisfy the orthogonality relations
 \begin{equation} \label{bprod} b_{ij}b_{ik}=0, \;\;\; b_{ji}b_{ki}=0 \tu{ for } k\in \bn \tu{ such that } j \neq k.  \end{equation}
 It is easy to see that this $\ast$-algebra admits many nontrivial representations on Hilbert spaces. For example, for any $n\in \bn$, we can
denote the canonical generators of $C(\mathbb{S}_n)$ by $\{q^{(n)}_{ij}:i,j=1,\ldots, n\}$ and put $b^{(n)}_{ij}=q^{(n)}_{ij}$
for $i,j \leq n$, $b^{(n)}_{ij}=0$ otherwise. Clearly,  $b^{(n)}_{ij}$ satisfy the required relations, so that we get a
$\ast$-homomorphism $\rho_n : {\mathcal B} \rightarrow C(\mathbb{S}_n)$ sending $b_{ij}$ to $b_{ij}^{(n)}$ and we can compose it
with any faithful representation of $C(\mathbb{S}_n)$. Since each $b_{ij}$ is a self-adjoint projection, the norm of its image
under any representation on a Hilbert space must be less than or equal to $1$. This implies that the universal norm defined by
$\| b \|:= \sup_{\pi} \| \pi(b) \| $, where $\pi$ varies over all representations
 of ${\mathcal B}$ on a Hilbert space, is finite. The completion of ${\mathcal B}$ under this norm will be denoted by ${\mathsf
 B}$. It is
  the universal $C^*$-algebra generated by $\{b_{ij}:i,j \in \bn\}$ satisfying relations \eqref{bproj}-\eqref{bprod}.
  We shall denote the universal enveloping von Neumann
 algebra of ${\mathsf B}$  by ${{\mathsf B}}^{**}$ and identify ${\mathsf B}$ as a $C^*$-subalgebra of ${{\mathsf B}}^{**}$.

Observe that for fixed $i\in \bn$, $p_i^{(n)}:=\sum_{j=1}^n b_{ij}$ is an increasing family of projections in ${\mathsf B}
\subset {{\mathsf B}}^{**}$, so it will converge in the ultraweak topology of  $\blg^{**}$ to some projection, say, $p_i$.
Similarly, for fixed $j\in \bn$, we write  $p^j:=\lim_{n\to \infty} \sum_{i=1}^n b_{ij}$
 in $\blg^{**}$. Let $w$ be the smallest central projection in $\blg^{**}$ which dominates $1-p_i, 1-p^j$ for all $i,j\in \bn$ and let $z=1-w$.
 Consider the $C^*$-algebra ${\mathsf A}:=z {\mathsf B}\subset {{\mathsf B}}^{**}$. Clearly, ${\mathsf A}$ is generated as a $C^*$-algebra
 by projections
 $\{q_{ij}:=zb_{ij}:i,j \in \bn\}$.
We claim that $(\alg, \{q_{ij}:i,j \in \bn\})$ is in $\Cclass$ and is indeed the universal $C^*$-algebra in this category.

First of all, it follows from the definition of $z$ that for each $i \in \bn$  we have $\sum_{j=1}^{\infty}
q_{ij}=1=\sum_{j=1}^{\infty} q_{ji}$ in the ultraweak topology inherited from the inclusion $z{{\mathsf B}}^{**} \subseteq
{{\mathsf B}}^{**}$, i.e. the ultraweak topology of ${\mathsf B}(z{\mathsf \Hil_u})$ where ${\mathsf H}_u$ denotes the universal
Hilbert space on which ${{\mathsf B}}^{**}$ acts. We complete the proof of the lemma by showing the universality of ${\mathsf A}$.
To this end, let ${\mathsf D}$ be a $C^*$-algebra generated by elements $\{t_{ij}:i,j \in \bn\}$ satisfying the relations
(\ref{infperm}), where the infinite series in (\ref{infperm}) converge in the ultraweak topology of the
 von Neumann algebra ${\pi({\mathsf D})}^{\prime \prime}$  for a fixed faithful representation $(\pi, \Hil)$ of ${\mathsf D}$.
 By  the definition of ${\mathsf B}$, we get
a $^*$-homomorphism from ${\mathsf B}$ onto ${\mathsf D}$ which sends $b_{ij}$ to $t_{ij}$ (for each $i,j \in \bn$). This
composed with $\pi$ extends to a unital, normal $\ast$-homomorphism, say $\rho$, from ${\mathsf B}^{**}$ onto ${\pi({\mathsf
D})}^{\prime \prime}$. In particular,  $\rho(p_i)=\sum_{j=1}^{\infty} t_{ij}=1$, and $\rho(p^i)=\sum_{j=1}^{\infty} t_{ji}=1$ for
all $i\in \bn$, so $1-p_i, 1-p^i $ belong to  the ultraweakly closed two-sided ideal ${\mathsf I}:={\rm Ker} \rho$ of ${{\mathsf
B}}^{**}$.
 Thus, if we denote by $w_0$ the central projection in $\blg^{**}$ such that ${\mathsf I}=w_0 {{\mathsf B}}^{**}$,  then $w_0$ dominates $1-p_i$ and
$1-p^i$ for all $i\in \bn$, and hence by the definition of $w$, we have $w_0 \geq w$. It follows that $w \in {\mathsf I}$, i.e.\
$\rho(w)=0$, or in other words, $\rho(z)=1$. This implies $\rho(b)=\rho(zb)$ for all $b \in {\mathsf B}$, so that we get a
$^*$-homomorphism $\rho_1:=\rho|_{\mathsf A}$ from ${\mathsf A}$ to ${\mathsf D}$ which satisfies $\rho_1(q_{ij})=t_{ij}$ for all
$i,j \in \bn$. This completes the proof of the universality of ${\mathsf A}$. \end{proof}

Denote the   von Neumann algebra  $z{{\mathsf B}}^{**}$ by   ${\mathsf A}_\infty$, and note that it should not be confused with
the universal enveloping von Neumann algebra of ${\mathsf A}$, which may be bigger. Note that the proof of the above theorem
indeed provides also a universal property of the von Neumann algebra ${\mathsf A}_\infty$, as stated in the next corollary.

\begin{cor}
\label{vnuniv} The von Neumann algebra ${\mathsf A}_\infty$ is the (unique up to  an isomorphism of von Neumann algebras)
universal object in the category of all von Neumann algebras ${\mathsf N}$ which are generated (in the ultraweak topology) by
projections $\{ n_{ij}: i,j \in \bn\}$ satisfying $\sum_{j=1}^{\infty}  n_{ij}=\sum_{j=1}^{\infty} n_{ji}=1_{\nlg}$ (convergence
in the ultraweak topology).
\end{cor}

Using the von Neumann algebraic universality we have the following result.
\begin{prop} \label{copAinf}
The von Neumann algebra ${\mathsf A}_\infty$  admits a natural coproduct $\Delta_{\alg}: \alg_{\infty} \to \alg_{\infty} \wot
\alg_{\infty}$ and a counit $\Cou_{\alg}: \alg_{\infty}\to \bc$.
\end{prop}
\begin{proof}
Consider for each $i,j \in \bn$ \[ x_{ij}:=\sum_{k=1}^{\infty} q_{ik} \ot q_{kj}\]
 as an element of ${\mathsf A}_\infty \wot
{\mathsf A}_\infty$. We note that the series converges in the ultraweak topology of the von Neumann algebra ${\mathsf A}_\infty
\wot {\mathsf A}_\infty$, the summands being mutually orthogonal projections. It is easy to check using the defining properties
of $q_{ij}$ that for each $i,j \in \bn$ there is $x_{ij}^2=x_{ij}=x_{ij}^*$, and $\sum_{k=1}^{\infty} x_{ik}=\sum_{k=1}^{\infty}
x_{ki}=1_{\alg_{\infty} \wot \alg_{\infty}}$. By the universality of the von Neumann algebra stated in Corollary \ref{vnuniv}, we
obtain a normal unital $\ast$-homomorphism $\Delta_{\alg} : {\mathsf A}_\infty \rightarrow {\mathsf A}_\infty \wot { {\mathsf
A}_\infty}$ given
 by  $\Delta_{\alg}(q_{ij})=x_{ij}, \;\;\; i,j\in \bn$, which is easily seen to be coassociative.  Similarly,
 we have a normal $^*$-homomorphism  $\epsilon_{\alg} : {\mathsf A}_\infty \rightarrow \bc$ given on generators by
$\epsilon_{\alg}(q_{ij})=\delta_{ij}$. Note that the existence of the counit implies in particular that $\Com_{\alg}$ is
injective.\end{proof}

The algebra $\alg_{\infty}$ is also equipped with a kind of an antipode.

\begin{prop}
The prescription \[ \kappa_{\alg}(q_{ij}) = q_{ji}, \;\;\; i,j \in \bn\] extends to a normal involutive $^*$-antihomomorphism of
$\alg_{\infty}$.
\end{prop}

\begin{proof}
View generators $q_{ij}$  as the elements of the opposite von Neumann algebra ${{\mathsf A}_\infty}^{\rm op}$ and denote them by
$\{q^o_{ij}:i,j \in \bn\}$. Once again using the universality as in Corollary \ref{vnuniv}, it is easy to see that the map
$q_{ij} \mapsto q^o_{ji}$ canonically induces a normal unital  $\ast$-homomorphism  from ${\mathsf A}_\infty$ to ${{\mathsf
A}_\infty}^{\rm op}$, which can be viewed as a $^*$- antihomomorphism on $\alg_{\infty}$.
\end{proof}

Let us now compare the construction above with that from the previous section. Recall the projective system
$(\wlg_n)_{n=1}^{\infty}$ of Hopf - von Neumann algebras introduced in Section \ref{winfty}. Let $\Wclass$ denote the
corresponding category of von Neumann algebras (as in Definition \ref{classes}).

\begin{prop} \label{propsummand}
The algebra $\alg_{\infty}$ of Corollary \ref{vnuniv} is an element of $\Wclass$. Therefore $\wlg_{\infty}$ is a direct summand
of $\alg_{\infty}$.
\end{prop}

\begin{proof}
Recall that $\alg_{\infty} \approx z \blg^{**}$ in the notation of Theorem \ref{construction}. The universal property of $\blg$
implies that for each $n \in \bn$ there is a surjection $\gamma_n: \blg \to C(\mathbb{S}_n)$ defined by the formula
\[\gamma_n (b_{ij})= \begin{cases} q_{ij}^{(n)} & i,j \leq n \\ 0 & \tu{otherwise} \end{cases}.\]
Let $\psi_n=\gamma_n^{**}$ - it again becomes a surjection, this time onto $\wlg_n=\QPn^{**}$, and it is easy to check that
$\psi_n = \phi_n \circ \psi_{n+1}$ for all $n \in \bn$. Hence $\blg^{**}$ is in the class $\Wclass$ associated with the sequence
$(\wlg_n)_{n \in \bn}$ according to Definition \ref{classes}.

 Define
$w_n$ to be the smallest central projection in $\blg^{**}$ dominating all projections $(p_j^{(n)})^{\perp}$ and $(q_j^{(n)})^{\perp}$, where
\[ p_j^{(n)} = \sum_{i=1}^n b_{ij}, \;\;\; q_j^{(n)} = \sum_{i=1}^n b_{ji}.\]
Note that we can describe $w_n$ in terms of the central supports of $(p_j^{(n)})^{\perp}$ and $(q_j^{(n)})^{\perp}$:
\begin{equation}
 \label{cent1}
w_n = \bigvee_{j \in \bn} z((p_j^{(n)})^{\perp}) \vee \bigvee_{j \in \bn} z((q_j^{(n)})^{\perp}).\end{equation} For that it
suffices to note that a central projection dominates another, not necessarily central, projection if and only if it dominates its
central carrier.

The argument similar to that of the proof of Theorem \ref{construction}, exploiting the fact that $\QPn^{**}$ can be described as
the universal von Neumann algebra generated by an $n$ by $n$ magic unitary implies that $\psi_n: w_n^{\perp}\blg^{**} \to
\QPn^{**}$ is an isomorphism. Indeed, it is easy to see that for each $j\in \bn$ there is $\psi_n(p_j^{(n)}) =
\psi_n(q_j^{(n)})=1_{\QPn^{**}}$, so that the projection determining the kernel of $\psi_n$ dominates $z_n:=w_n^{\perp}$ and
$\psi_n(x)=\psi_n(z_n x)$ for all $x \in \blg^{**}$. Thus we obtain a surjective map $\psi_n|_{z_n \blg^{**}}\to \QPn^{**}$ which
preserves the natural magic unitaries in both algebras (observe that $\sum_{i=1}^n z_n b_{ij} = \sum_{i=1}^n z_n b_{ji}=z_n$).
The afore-mentioned universality of $\QPn^{**}$ implies that it is an isomorphism.

Hence $\Ker(\psi_n)$ is equal to $w_n \blg^{**}$ and the intersection $\bigcap_{n \in \bn}\Ker(\psi_n)$ is equal to
$w_{\infty}\blg^{**}$, where $w_{\infty}=\lim_{n \in \bn} w_n$.

Recall that the central projection $w=z^{\perp} \in \blg^{**}$ was defined in the proof of Theorem \ref{construction} as the
smallest central projection in $\blg^{**}$ dominating all projections $p_j^{\perp}$ and $q_j^{\perp}$, where $p_j = \lim_{n \in
\bn} p_j^{(n)}$ and $q_j = \lim_{n \in \bn} q_j^{(n)}$. Hence it is easy to check that $z\geq z_{\infty}:=w_{\infty}^{\perp}$ and
in particular we can view $z\blg^{**}$ as an element of $\Wclass$.
\end{proof}

The inclusion $\wlg_{\infty} \subset \alg_{\infty}$ is close to being an inclusion of Hopf - von Neumann algebras. This is
formulated in the next proposition.

\begin{prop}
View $\wlg_{\infty}$ as a subalgebra of $\alg_{\infty}$, so that $\wlg_{\infty} = z_{\infty} \alg_{\infty}$. The normal
$^*$-homomorphism $\hat{\Com}: \wlg_{\infty} \to \wlg_{\infty} \wot \wlg_{\infty}$ defined by: $\hat{\Com}(x) = (z_{\infty} \ot
z_{\infty}) (\Com_{\alg}(x))$ ($x \in \wlg_{\infty}$) is unital and coassociative. It in fact coincides with the coproduct on
$\wlg_{\infty}$ constructed as a projective limit in Theorem \ref{projcop}.
\end{prop}

\begin{proof}
We use the notation of the last proposition. As $\wlg_{\infty}=z_{\infty} \alg_{\infty}$, it is enough to show that
$\Com_{\alg}(z_{\infty}) \geq z_{\infty} \otimes z_{\infty}$, so that $\hat{\Com}: \wlg_{\infty} \to \wlg_{\infty} \wot
\wlg_{\infty}$ satisfies the required conditions.

Note first that as $\Ker(\psi_n)=w_n \alg_{\infty}$, we can check that \[\Ker(\psi_n \otimes \psi_n) = (z_n \ot
z_n)^{\perp}(\alg_{\infty} \wot \alg_{\infty}).\]  The construction of the coproduct on $\alg_{\infty}$ implies that the maps
$\psi_n: \alg_{\infty} \to \QPn^{**}$ intertwine the respective coproducts (recall that $\QPn^{**}$ has a canonical Hopf--von
Neumann algebra structure induced from $\QPn$). As we have $(\psi_n \ot \psi_n) (\Com_{\alg}(w_n))= \Com_n (\psi_n(w_n))=0$, the formula
displayed  above implies that the projection $\Com_{\alg}(w_n)$ is dominated by $(z_n \ot z_n)^{\perp}$. Passing to the
limit (exploiting normality of the coproduct) we obtain that
\[ \Com_{\alg}(w_{\infty}) \leq (z_{\infty}  \ot z_{\infty})^{\perp},\]
so the proof of the first statement of the lemma is finished.

To show the second part, by the uniqueness in Theorem \ref{projcop} it suffices to show that for each $n \in\bn$ we have
\[\Com_n  \psi_n|_{\wlg_{\infty}} = (\psi_n|_{\wlg_{\infty}} \ot \psi_n|_{\wlg_{\infty}}) \hat{\Com}.\]
In fact we can even show that
\begin{equation} \label{equalcop}\Com_n  \psi_n= (\psi_n \ot \psi_n) \Com_{\alg}.\end{equation}
Indeed, as maps on both sides of the last equation are normal, it suffices to check they take the same values on each $zb_{ij}$
(where $z$ is now a central projection in $\blg^{**}$ defined in Theorem \ref{construction}). Fix then $i,j \in \bn$:
\begin{align*}
 (\psi_n \ot \psi_n)& (\Com_{\alg} (zb_{ij})) = (\psi_n \ot \psi_n) (\lim_{k\to\infty} \sum_{l=1}^k z b_{il} \ot z b_{lj}) \\&= \lim_{k\to\infty}  (\psi_n \ot \psi_n) (\sum_{l=1}^k z b_{il} \ot z b_{lj})
= \sum_{l=1}^n  \psi_n (zb_{il}) \ot \psi_n (zb_{lj}).
\end{align*}
Now it is easy to check that $\Com_n (\psi_n(zb_{ij}))= (\psi_n \ot \psi_n) (\Com(zb_{ij}))$, considering separately two cases:
first $i,j\leq n$ and then $\max\{i,j\}>n$. Thus \eqref{equalcop} is proved.
\end{proof}

Proposition  \ref{propsummand} does not exclude the possibility of $\alg_{\infty}$ actually coinciding with $\wlg_{\infty}$,
i.e.\ $z=z_{\infty}$. Below we show that this is not the case.

\begin{lem}
Let $z, z_{\infty} \in \Proj(\blg^{**})$ be the projections introduced in the proof of Proposition \ref{propsummand}. Then $z\neq z_{\infty}$.
\end{lem}

\begin{proof}
Observe that  another application of the argument used in Proposition \ref{propsummand} implies that
\begin{equation}
 \label{cent2}
z^{\perp}= \bigvee_{j \in \bn} z(p_j^{\perp}) \vee \bigvee_{j \in \bn} z(q_j^{\perp}),\end{equation} so the comparison of the
formulas \eqref{cent1} and \eqref{cent2} shows that the problem of deciding whether $z=z_{\infty}$ is related to the fact that
for a decreasing sequence of projections in a von Neumann algebra, say $(r_n)_{n=1}^{\infty}$ we can have $z(\lim_{n \in \bn}r_n)
\neq \lim_{n \in \bn}\; z(r_n)$.

Suppose for the moment that there exists a non-zero normal representation $\pi:\blg^{**} \to B(\hil)$ such that $\pi(z)
=1_{B(\hil)}$, $\nlg:=\pi(\blg^{**})$ is a factor, and if we write $d_{ij}=\pi(b_{ij})$ ($i,j \in \bn$) then we have
$r_k:=\sum_{j=1}^k d_{1j} \neq 1_{B(\hil)}$ for all $k \in \bn$. Then $z(r_k^{\perp})= 1_{\nlg}=1_{B(\hil)}$ (central carrier
understood in $\nlg$). As $\pi: \blg^{**}\to \nlg$ is onto (so in particular it maps $Z(\blg^{**})$ into $Z(\nlg)$), we have for
each $p \in \Proj(\blg^{**})$ the inequality $z(\pi(p)) \leq \pi(z(p))$. As $r_k^{\perp}=\pi((p_1^{(k)})^{\perp})$, we have
therefore (recall (\eqref{cent1})
\[\pi(z_k^{\perp}) \geq \pi(z((p_1^{(k)})^{\perp}))\geq z(r_k^{\perp})=1_{B(\hil)}.\]
Hence $\pi(z_k)=0$ and thus also $\pi(z_{\infty})=0$, so $z$ cannot be equal to $z_{\infty}$.

It remains to show that such a representation exists. It suffices to exhibit a concrete magic unitary $(d_{ij})_{i,j=1}^{\infty}$
built of projections on a Hilbert space $\hil$  such that each row and column sum to $1_{B(\hil)}$, $\sum_{j=1}^k d_{1j} <
1_{B(\hil)}$ for each $k \in \bn$ (in other words the first row is not `finitely supported') and the entries generate $B(\hil)$
as a von Neumann algebra. Let then $(d_n)_{n=1}^{\infty}$ be a sequence of non-zero mutually orthogonal projections summing to
$1_{B(\hil)}$ and consider the matrix:
\[ \begin{bmatrix}  d_1 & 0 & d_2 & d_3  &d_4 & \cdots \\
d_1^{\perp} & d_1 & 0 & 0 & 0 & \cdots \\
0& d_2  & d_2^{\perp} & 0 & 0 & \cdots \\
0&d_3 & 0 & d_3^{\perp}  & 0 & \cdots \\
0&d_4 & 0 & 0&  d_4^{\perp}  & \cdots \\
\vdots&\vdots & \vdots & \vdots  & \vdots & \vdots
\end{bmatrix}.
\]
It is easy to see it gives a magic unitary with the first row `infinitely supported'. The generation condition can be  achieved
by considering a finite sequence of projections $(t_n)_{n=1}^{k}$ generating the whole $B(\hil)$ and adding to a given magic
unitary two by two blocks of the form $\begin{bmatrix} t_n & t_n^{\perp} \\ t_n^{\perp}  & t_n\end{bmatrix}$ (with respective
rows and columns completed by zeros).
\end{proof}

\begin{cor}$\wlg_{\infty}$ is a proper von Neumann subalgebra of $\alg_{\infty}$.
\end{cor}

\vspace*{0.2cm} \noindent \textbf{Acknowledgment.} Part of the work on this paper was done during the visit of the second author to ISI Kolkata in April 2010 funded by the  UKIERI
project Quantum Probability, Noncommutative Geometry and Quantum Information. The first author acknowledges the support of the project `Noncommutative Geometry and Quantum Groups' (part of Swarnajayanti Fellowship and Project) funded by Department of Science and Technology, Govt.\ of India. The second author was partially supported by National Science Center (NCN) grant no.~2011/01/B/ST1/05011.


\begin{thebibliography}{Debashish}








\bibitem[Ban$_1$]{Kesten} T.\,Banica, Representations of compact quantum groups and subfactors, \emph{J.\,Reine Angew.\,Math.} \textbf{509} (1999), 167�-198.

\bibitem[Ban$_2$]{metric} T.\,Banica, Quantum automorphism groups of small metric spaces,
 \emph{Pacific J.\,Math.} \textbf{219}  (2005),  no.\,1, 27--51.


\bibitem[Ban$_3$]{surveyTeo} T.\,Banica, Quantum permutations, Hadamard matrices, and the search for matrix models, \emph{preprint}, available at arxiv:1109.4888.

\bibitem[BBC]{Teosurvey} T.\,Banica, J.\,Bichon and B.\,Collins,
Quantum permutation groups: a survey,  \emph{Banach Center Publ.} \textbf{78} (2007), 13--34.



\bibitem[BaS]{orthfiltr} T.\,Banica and A.\,Skalski, Quantum symmetry groups of $C^*$-algebras equipped with orthogonal filtrations, \emph{preprint}, available at arxiv:1109.6184.

\bibitem[Bic]{graph} J.\,Bichon, Quantum automorphism groups of finite graphs,
\emph{Proc.\,Am.\,Math.\,Soc.} \textbf{131}  (2003),  no.\,3, 665--673.







\bibitem[Cur]{Curran} S.\,Curran, Quantum exchangeable sequences of algebras, \emph{Indiana Univ.\,Math.\,J.}  \textbf{58}  (2009),  no.\,3, 1097--1125.

\bibitem[Gos]{Deb} D.\,Goswami, Quantum Group of Isometries in Classical and Noncommutative Geometry,
\emph{Comm.\,Math.\,Phys.} \textbf{285} (2009), no.\,1,  141--160.

\bibitem[Gui]{Guichardet} A.\,Guichardet, Sur la cat\'egorie des algebres de von Neumann, \emph{Bull.~Sci.~Math.\ (2)} \textbf{90} (1966), 41–-64.


\bibitem[KSp]{ClausRoland} C.\,K\"ostler and R.\,Speicher, A noncommutative de Finetti theorem: invariance under quantum permutations is equivalent
to freeness with amalgamation, \emph{Comm.\,Math.\,Phys.} \textbf{291}  (2009),  no.\,2, 473--490.

\bibitem[KuV]{kuv}
J.\,Kustermans and S.\,Vaes, Locally compact quantum groups, \emph{Ann.\,Sci.\,\'Ecole Norm.\,Sup. (4)} \textbf{33} (2000)
no.\,6, 837--934.

\bibitem[MVD]{VanDaele} A.\,Maes and A.\,Van Daele, Notes on compact quantum groups, \emph{Niew Arch.Wisk(4)} \textbf{16} (1998), no.\,1-2, 73--112.

\bibitem[MaV]{Varghese} S.\,Mahanta and M.\,Varghese, Operator algebra quantum homogeneous spaces of universal gauge groups,  \emph{Lett.~Math.~Phys.} \textbf{97} (2011) 263--277.

\bibitem[SeZ]{Semadeni} Z.~Semadeni and H.~Zidenberg,  Inductive and inverse limits in the category of Banach spaces, \emph{Bull.~Acad.~Polon.~Sci.~S\'er.~Sci.~Math.~Astronom.~Phys.} \textbf{13} (1965), 579–-583.

\bibitem[Wan]{Wang} S.\,Wang, Quantum symmetry groups of finite spaces, \emph{Comm.\,Math.\,Phys.} \textbf{195}
(1998),  no.\,1, 195--211.



\bibitem[Wor]{woronowicz98}
S.L. Woronowicz, Compact quantum groups, \emph{in} A.~Connes, K.~Gawedzki, and J.~Zinn-Justin, editors, {\em
  Sym{\'e}tries Quantiques, Les Houches, Session LXIV, 1995},  845--884.






\end{thebibliography}
\end{document}